\documentclass{amsart}
\usepackage{graphicx}
\title{Whitney's formulas for curves on surfaces}
\author{Yurii Burman}
\address{121002, Independent University of Moscow, 11, B.
Vlassievsky per., Moscow, Russia} \email{burman@mccme.ru}
\author{Michael Polyak}
\address{Department of Mathematics, Technion- Israel
Institute of Technology, Haifa 32000, Israel}
\email{polyak@math.technion.ac.il}
\date{}

\makeatletter

\RequirePackage{amssymb}

\newtheorem{thm}{Theorem}
\newtheorem{prop}[thm]{Proposition}
\newtheorem{cor}[thm]{Corollary}
\newtheorem{lem}[thm]{Lemma}

\theoremstyle{definition}

\newtheorem{ex}[thm]{Example}

\newtheorem{rem}{Remark}

\def \g {\gamma}
\def \tg {\gamma\,'}

\def \gO {\Omega}
\def \S {\Sigma}
\def \uS {\widetilde{\Sigma}}
\def \R {{\mathbb R}}
\def \ind {\name{ind}}
\def \sminus{\smallsetminus}

\def \Z {{\mathbb Z}}

\def \T {{\mathbb T}}

\def \lmod#1\rmod {\vphantom{#1}\left|\smash{#1}\right|}

\newcommand \bydef {\stackrel{\mbox{\scriptsize def}}{=}}

\newcommand{\pder}[2] {\frac{\partial #1}{\partial #2}}

\@ifundefined{name}{\newcommand \name[1] {\mathop{\rm #1}\nolimits}}

\newcommand \eps {\varepsilon}
\renewcommand \phi {\varphi}
\renewcommand \rho {\varrho}

\makeatother

\def \sgn {\name{sgn}}
\def \Traj {\Phi}
\def \windex {\name{w}}

\def \Ann {\name{K}^\eps}
\def \indT {\ind_T(\g,p)}
\def \gT {\langle \g \rangle_T}
\def \indeps {\ind_\eps(\g,p)}
\def \geps {\langle \g \rangle_\eps}
\def \ginfty {\langle \g \rangle_\infty}
\def \ev {\name{ev}}
\def \I {\operatorname{I}}
\def \proj {\name{proj}}

\def\fig#1#2#3#4
 {\begin{figure}[htb]\includegraphics[#1]{#2}\caption{#3}\label{#4}\end{figure}}

\begin{document}

 \begin{abstract}
The classical Whitney formula relates the number of times an
oriented plane curve cuts itself to its rotation number and the
index of a base point. In this paper we generalize Whitney's formula
to curves on an oriented punctured surface $\S_{m,n}$, obtaining a
family of identities indexed by elements of $\pi_1(\S_{m,n})$. To
define analogs of the rotation number and the index of a base point
of a curve $\g$, we fix an arbitrary vector field on $\S_{m,n}$.
Similar formulas are obtained for non-based curves.
 \end{abstract}



\thanks{The first author was supported by the Scientific foundation of the HSE project 09-01-0015 ``Hurwitz generating functions and embedded graphs'', by the CRDF grant RUM1-2895-MO-07 ``Rational Cherednik algebras, inverse Macaulay systems, and complete integrability'', by the RFBR grants 08-01-00110-a and N.Sh.-709.2008.1 (``Scientific school of V.I.Arnold''). The second author was partially supported by the ISF grant 1261/05.}

\maketitle

\section{Introduction}

In this paper we study self-intersections of smooth immersed curves
on an oriented surface $\S=\S_{m,n}$ of genus $m$ with $n$
punctures. Fix $p\in\S$ and denote $\pi=\pi_1(\S,p)$. We will
consider immersions $\g: [0,1] \to \S$ with $\g(0) = \g(1)= p$ and
$\tg(0) = \tg(1)$ --- we call them {\em closed curves with the base
point $p$}. Throughout the paper, all curves are assumed to be
generic, i.e. their only singularities are double points of
transversal self-intersection, distinct from $p$.

\subsection{Self-intersections of an immersed curve}
\label{Sec:Selfintersec}
Let $\g$ be a closed curve and $d$ be a self-intersection point $d =
\g(u) = \g(v)$, $u<v$. Define $\sgn(d) = +1$ if the orientation of
the basis $(\tg(u),\tg(v))$ coincides with the one prescribed by the
orientation of $\S$, and $\sgn(d) = -1$ otherwise, see Figure
\ref{Fig:Signs}a. Turaev \cite{Turaev} constructed an important
element of the group ring $\Z[\pi]$ corresponding to $\g$, in the
following way. Let $\tau_d(\g)\in\pi$ be the homotopy class of a
loop $\g(t)$ with $t \in [0,u] \cup [v,1]$. Denote $D(\g)$ the set
of double points of $\g$ and define the element
$\langle\g\rangle\in\Z[\pi]$ by
 \begin{equation}\label{Eq:Sum}
\langle \g \rangle=\sum_{d\in D(\g)}\sgn(d)\tau_d(\g).
 \end{equation}
In particular, for a curve on $\S=\R^2$ one has $\Z[\pi]=\Z$ so
$\langle\g\rangle=\sum_{d\in D(\g)}\sgn(d)\in\Z$.

 \begin{ex}\label{Ex:Radial1}
Let $\S=\S_{0,2}=\R^2\sminus\{0\}$ and denote by $g$ the generator of $\pi$
(represented by a small loop around $0$). For a curve $\g$ shown in Figure
\ref{Fig:Radial}a we have $\langle\g\rangle=g^2-g$; signs of
self-intersections and the corresponding curves $\tau_d$ are shown in
Figure \ref{Fig:Radial}b.
 \end{ex}

For a curve $\g$ on $\S=\R^2$ one may define its {\em Whitney index} (or
winding number) $\windex(\g)$ as the number of full rotations made by the
tangent vector $\tg(t)$ around the origin, as $t$ moves from $0$ to $1$.

For a curve $\g$ on $\S=\R^2$ and $x\in\R^2\sminus\g$ one can also
define $\ind(\g,x)$ as the number of times the curve $\g$ circles
around $x$. In other words, it is the linking number of a $1$-cycle
$[\g]$ with the $0$-chain $[\infty]-[x]$ (composed of a point near
infinity taken with the positive sign and $x$ taken with the
negative sign). It can be calculated as the intersection number of
the curve $\g$ with any ray starting in $x$ and going to infinity.
If $\g$ is a plane curve with a base point $p$, we define
$\ind(\g,p)\in\frac{1}{2}\Z$ by averaging the values of $\ind$ on
two components of $\R^2\sminus \g$ adjacent to $p$.

H.\,Whitney in \cite{Whitney} considered closed curves with a base point on
$\S=\R^2$ and showed that

 \begin{thm}[\cite{Whitney}] \label{Th:Whitney}
Let $\g: [0,1] \to \R^2$ be a generic immersed curve with the base point $p
= \g(0) = \g(1)$. Then
 \begin{equation}\label{Eq:Whitney}
\langle \g \rangle = - \windex(\g) + 2\ind(\g,p).
 \end{equation}
 \end{thm}

\fig{width=5in}{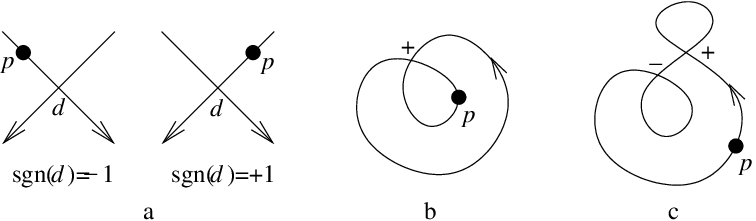}{Signs of self-crossings and some simple curves on
$\R^2$.}{Fig:Signs}

 \begin{ex}
For a curve $\g$ shown on Figure \ref{Fig:Signs}b we have $\langle
\g\rangle = +1$, $\windex(\g)=2$, and $\ind(\g,p)=3/2$. For a curve $\g$
shown on Figure \ref{Fig:Signs}c we have $\langle \g\rangle = +1-1 = 0$,
$\windex(\g)=1$, and $\ind(\g,p)=1/2$.
 \end{ex}

\subsection{Curves on surfaces}
 \label{Sub:Surf}
The main results of this paper are generalizations of
Theorem~\ref{Th:Whitney} for curves on surfaces. We define appropriate
surface versions of expressions in both sides of \eqref{Eq:Whitney} and
relate them.

To define an analog of the Whitney index for a curve $\g: S^1 \to \S$, we
fix a vector field $X$ on $\S$ having no zeros on the curve $\g$. Define,
following \cite{Chilling, Reinhart} $w(\g,X)$ to be the number of rotations
of $\tg(t)$ relative to $X$. It can be calculated as the algebraic number
of points in which $\tg(t)$ looks in the direction of $X$; each such point
is counted with a positive sign, if $\tg$ turns counter-clockwise relative
to $X$ in a neighborhood of $\g(t)$ and with a negative sign otherwise.
It is clear that $w(\g, X)$ does not change under the homotopies of $\g$
and $X$ (as long as $\g$ stays an immersion and $X$ has no zeros on its image). Define the {\em Whitney index} of $\g$ as $\windex(\g,X) \bydef w(\g,X) [\g] \in
\Z[\pi]$ where $[\g]$ is the class of homotopy represented by the curve $\g$.


In particular, if $\S=\S_{0,1}=\R^2$ and $X=\pder{}{x}$ is a horizontal
vector field directed from left to right, then $\pi$ is trivial and
$\windex(\g,X)=\windex(\g)$ is the usual Whitney index of $\g$.

\fig{width=5in}{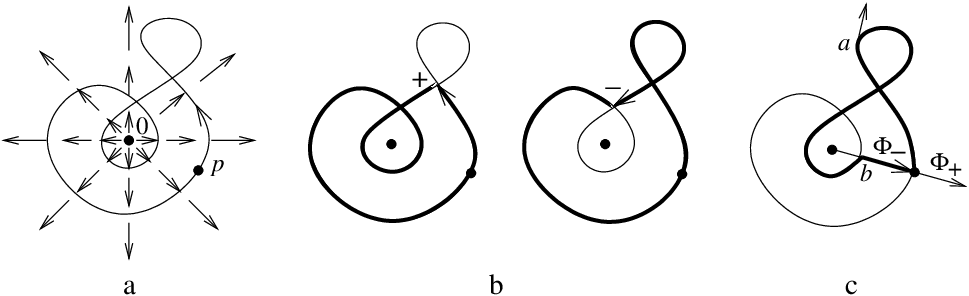}{A curve on a punctured plane and a radial
vector field.}{Fig:Radial}

 \begin{ex}\label{Ex:Radial2}
Let's return to Example \ref{Ex:Radial1}. Let $X =
f(x,y)(x\pder{}{x}+y\pder{}{y})$ be a radial vector field on $\S =
\S_{0,2} = \R^2\sminus\{0\}$; the function $f(x,y)>0$ is chosen so
that the field $X$ is smooth. One may check that
$\windex(\g,X)=(\windex(\g)-k)g^k$, where $\windex(\g)$ is the usual
Whitney index of $\g$, $k=\ind(\g,0)$ is the number of times $\g$ circles
around $\{0\}$, and $g$ is the generator of $\pi =
\pi_1(\R^2\sminus\{0\})$. In particular, for the curve $\g$ shown on
Figure \ref{Fig:Radial}a, the number of rotations of $\tg$ relative to
$X$ is $-1$ (indeed, there is only one positive tangency point of $X$
with $\g$, denoted by $a$ in Figure \ref{Fig:Radial}c; its sign is $-1$).
Also, $[\g]=g^2$, thus $\windex(\g,X)=-g^2$.
 \end{ex}

The results obtained in this article generalize those of the papers
\cite{Polyak} and \cite{R2T2} where the curves on $\S_{0,1}=\R^2$ and
$\S_{1,0}=\T^2$ were considered; the methods we use are similar to those of
\cite{R2T2}.

We consider both curves with a base point (Theorem \ref{Th:Finite}),
and curves without base points (Theorem~\ref{Th:NBFinite}). Both
theorems are proved by similar methods, so we collected all proofs
in Section~\ref{Sec:Proofs}.

\section{Curves with base points}\label{Sec:Curves}

\subsection{Making loops from pieces}\label{Sub:Compose}
For two generic paths $\g_i:[a_i,b_i]\to\S$, $i=1,2$ with $\g_1(a_1)
= \g_2(b_2) = p$ we may define an element similar to \eqref{Eq:Sum}
as follows. Let $d = \g_1(u) = \g_2(v)$ be an intersection point of
$\g_1$ and $\g_2$, with $u \ne a_1,b_1$, $v \ne a_2,b_2$. Again,
define $\sgn(d) = +1$ if the orientation of the basis
$(\tg_1(u),\tg_2(v))$ coincides with the one prescribed by the
orientation of $\S$, and $\sgn(d) = -1$ otherwise. Let
$\tau_d(\g_1,\g_2)\in\pi$ be the homotopy class of a loop composed
of two arcs: $\g_1(t_1)$ with $t_1 \in [a_1,u]$, and $\g_2(t_2)$
with $t_2 \in [v,b_2]$, see Figure \ref{Fig:Def}a.

\fig{height=1.4in}{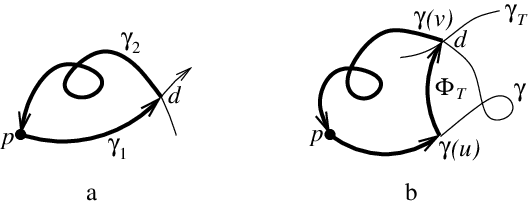}{Making a loop out of two and three
pieces.}{Fig:Def}

Denote by $D(\g_1,\g_2)$ the set of intersections of $\g_1$ and
$\g_2$ and define $\langle \g_1,\g_2 \rangle \in\Z[\pi]$ by
 \begin{equation}\label{Eq:SumGG}
\langle \g_1,\g_2 \rangle= \sum_{d\in
D(\g_1,\g_2)}\sgn(d)\tau_d(\g_1,\g_2).
 \end{equation}

\subsection{Intersections and $\tau$-indices}
 \label{Sub:Finite}
Let $X$ be a vector field on the surface $\S$. Throughout the paper,
we will assume that every trajectory of $X$ is infinitely
extendable, so a one-parametrical group $\Traj_t$ of diffeomorphisms
generated by $X$ is well-defined. Denote by $\Traj(a)$ its integral
trajectory starting at the point $a \in \S$.

Given $X$, we generalize formula \eqref{Eq:Sum} as follows. Let $\g:
[0,1] \to \S$ be a generic immersed curve with the base point $p =
\g(0) = \g(1)$. Fix $T>0$ such that $\Traj_{\pm T}(p)\notin\g$.
Denote by $\Traj_-$ and $\Traj_+$ the negative (resp. positive)
$T$-time semi-trajectories $\Traj_t(p)$, $t \in [-T,0]$ (resp. $t
\in [0,T]$) of the base point $p$.

Define the {\em $T$-time index of $p$ with respect to $\g$} by
 \begin{equation}\label{Eq:DefInd}
\indT= \frac12 \bigl( \langle \g, \Traj_- \rangle + \langle \Traj_+,
\g \rangle \bigr) \in \frac12\Z[\pi]
 \end{equation}
Denote by $\g_T: [0,1] \to \S$ a $T$-shift of $\g$ along $X$:
$\g_T(t)=\Traj_T(\g(t))$. Suppose that $X$ does not vanish on $\g$
and all the intersections of $\g_T$ with $\g$ are transversal double
points $d=\g_T(u)=\g(v)$. Define $\sgn(d) = +1$ if the orientation
of the basis $(\tg_T(u),\tg(v))$ coincides with the orientation of
$\S$, and $\sgn(d) = -1$ otherwise. Let $\tau_{d,T}(\g)\in\pi$ be
the homotopy class of a loop composed of three arcs: $\g(t_1)$ with
$t_1 \in [0,u]$, $\Traj_{t_2}(\g(u))$ with $t_2 \in [0,T]$, and
$\g(t_3)$ with $t_3 \in [v,1]$, see Figure \ref{Fig:Def}b.

Between all such intersection points $d$ pick only the ones with
$u<v$; denote this set by $D_T(\g)$. Define the element
$\gT\in\Z[\pi]$ by
 \begin{equation}\label{Eq:SumT}
\gT=\sum_{d\in D_T(\g)}\sgn(d)\tau_{d,T}(\g).
 \end{equation}
In particular, for a curve on $\S=\R^2$ one has $\Z[\pi]=\Z$ so
$\gT\in\Z$.

For a small $T=\eps$ intersections $\g(u)=\g_\eps(v)$ appear near double
points of $\g$ and near points in which $\g$ is tangent to $X$ (the
condition $u<v$ means in this case that the tangent vector should be
directed in the direction of $X$). After checking the signs, we see that
the contribution of double points to $\geps$ equals $\langle \g \rangle$,
while the contribution of points in which $\g$ is tangent to $X$ equals
$\windex(\g,X)$ (recall that we consider $\windex(\g,X)$ as a multiple
of the class $[\g]$) and thus obtain
 \begin{equation}\label{Eq:epsilon}
\geps=\langle \g \rangle+\windex(\g,X)
 \end{equation}

 \begin{thm}\label{Th:Finite}
Let $\g: [0,1] \to \S$ be a generic immersed curve with the base
point $p = \g(0) = \g(1)$. Let $X$ be a vector field which does not
vanish on $\g$ and is transversal to $\g$ at the base point $p$.
Suppose that $T>0$ is such that all intersections of $\g$ with
$\g_T$ are transversal double points distinct from $p$. Then
 \begin{equation}\label{Eq:Finite}
\langle \g \rangle = \gT - \windex(\g,X) + 2\indT.
 \end{equation}
 \end{thm}

 \begin{rem}
For a small $T=\eps$ we have $\indeps=0$ and $\geps$ is given by
\eqref{Eq:epsilon}, so in this case equality \eqref{Eq:Finite} is
trivial.
 \end{rem}

 \begin{ex}
Let $\S=\S_{0,1}=\R^2$ and $X=\pder{}{x}$ be a horizontal vector
field directed from left to right. Here $\pi$ is trivial, and
$\windex(\g,X)=\windex(\g)$ is the usual Whitney index of $\g$. For
a large $T$, $\g_T$ is shifted far from $\g$, so $\gT=0$. Also,
$\indT$ counts intersections of $\g$ with both horizontal rays
emanating from $p$ to the left and to the right. So it equals to the
index $\ind(\g,p)$ of the base point, introduced in Section
\ref{Sec:Selfintersec}. Equation \eqref{Eq:Finite} turns in this
case into the classical Whitney's formula \eqref{Eq:Whitney}.
 \end{ex}

 \begin{ex}\label{Ex:Radial3}
Let's return to Example \ref{Ex:Radial2}. For a large $T$ the curve
$\g_T$ is shifted far from $\g$, so $\gT=0$. The curve $\g$ shown in
Figure \ref{Fig:Radial}a has no intersections with $\Traj_+$, and
only one intersection with $\Traj_-$, denoted by $b$ in Figure
\ref{Fig:Radial}c; its sign is $-1$. The corresponding curve
$\tau_b$ is depicted there in bold. It represents a class $g$, thus
$2\indT=-g$. This agrees with
$\langle\g\rangle+\windex(\g,X)=(g^2-g)-g^2=-g$.
 \end{ex}

\subsection{Formula for an infinite time shift}\label{Sub:Infinite}
Let us consider the behavior of formula \eqref {Eq:Finite} when $T
\to \infty$.

\begin{prop}\label{Th:Infinite}
Let $\g: [0,1] \to \S$ be a generic immersed curve with the base
point $p = \g(0) = \g(1)$. Let $X$ be a vector field which does not
vanish on $\g$. Suppose that the trajectory $\Traj(p)$ of the base
point intersects $\g$ only in a finite number of points. Then limits
$\displaystyle{\ind(\g,p)=\lim_{T\to\infty}\indT}$ and
$\displaystyle{\ginfty=\lim_{T\to\infty}\gT}$ are well defined and
satisfy
 \begin{equation}\label{Eq:Infinite}
\langle \g \rangle = \ginfty - \windex(\g,X) + 2\ind(\g,p).
 \end{equation}

\end{prop}

\begin{proof}
By the assumption, $\Traj(p)$ intersects $\g$ in a finite number of
points $\Traj_{t_i}(p)$, $i=1,2,\dots,k$ with $t_1<t_2<\dots<t_k$.
Thus when $T>\max(-t_1,t_k)$, the index $\indT$ does not change and
$\lim_{T\to\infty}\indT$ is well-defined.

Now the statement follows from Theorem \ref{Th:Finite}. Indeed, note
that in addition to $\indT$, only one term in equation
\eqref{Eq:Finite} depend on $T$, namely $\gT$. Since equation
\eqref{Eq:Finite} is satisfied for any $T$, we conclude that $\gT$
also does not change when $T>\max(-t_1,t_k)$. Therefore,
$\lim_{T\to\infty}\gT$ is also well-defined. Since the equality
\eqref{Eq:Finite} holds for any $T$, it is satisfied also in the
limit $T \to \infty$.
\end{proof}

\begin{rem}
As an example of how $\g_T$ behaves as $T\to\infty$ consider
$\S=\S_{m,0}$ and let $X$ be a gradient vector field of a (general
position) Morse function on $\S$. This vector field has $2g+2$ zeros
corresponding to $2g$ saddle points, one maximum, and one minimum of
the function. Critical points of neighboring indices are connected
by trajectories of $X$, called separatrices. See Figure
\ref{Fig:Morse}. Each saddle point $q_i$, $i=1,\dots,2g$, is joined
to the maximum by a pair of separatrices, making a loop $l_i$.
Homotopy classes of the loops $l_1,\dots,l_{2g}$ form a system of
generators in $\pi$.

\fig{height=1.6in}{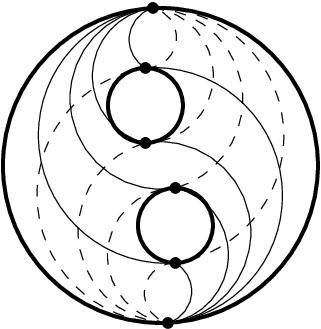}{Separatrices of a gradient field.}{Fig:Morse}

As $T \to \infty$, the curve $\g_T$ is attracted into a neighborhood
of $l_1\cup\dots\cup l_{2g}$.
So one can calculate $\ginfty$ in another way, using intersections of
$\g$ with the separatrices.

This method of calculating $\ginfty$ may be applied to a wider class
of vector fields as well. Separatrices of a field are then trajectories
which bound regions of different dynamics of the flow on $\S$.
Under appropriate assumptions on the field the curve $\g_T$ is
attracted, as $T \to \infty$, into a neighborhood of the union of
separatrices. Therefore $\ginfty$ may be calculated from the
intersections of $\g$ with them.
However, for a general vector field this approach involves a number of
technicalities and requires a lengthy treatment, so we omit it here.
\end{rem}

\section{Curves without base points}\label{SSec:NoBase}
Our goal is to repeat all constructions of Section \ref{Sec:Curves}
for curves without base points. In all formulas we will be using now
free loops instead of based loops. Denote by $\gO$ the space of
homotopy classes of free loops on $\S$.
Let $\g: S^1 \to \S$ be an oriented curve without a base point;
we are to define an appropriate analogue of $\langle \g \rangle$
in this case. Let $d$ be a self-intersection of $\g$.
Smoothing $\g$ in $d$ with respect to the orientation, we obtain two
closed curves: $\g^l_d$ and $\g^r_d$, where the tangent vector of
$\g^l_d$ rotates clockwise in the neighborhood of $d$, and that of
$\g^r_d$ rotates counter-clockwise. See Figure \ref{Fig:NBased}a.
Denote $D(\g)$ the set of double points of $\g$ and define $\langle
g \rangle \in\Z[\gO]$ by
 \begin{equation*}
\langle \g \rangle=\sum_{d\in D(\g)}([\g^l_d]-[\g^r_d]).
 \end{equation*}
\fig{width=5in}{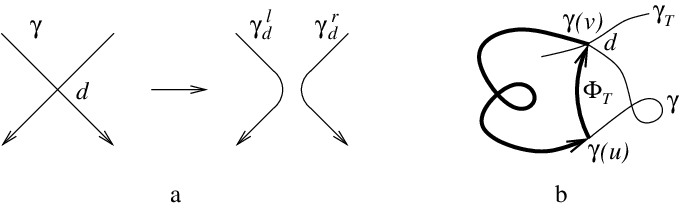}{Smoothing a crossing and constructing
loops from pieces}{Fig:NBased}

Define the Whitney index in the non-based case as $\windex(\g,X) \bydef
w(\g,X)([\g]-1) \in \Z[\gO]$ where $1$ is the class of the trivial loop,
and $w(\g,X)$ is, like for pointed curves, the number of full rotations of
$\tg(t)$ relative $X$. (Note that the definition of $w(\g,X)$ does not
require the base point.)

The definition of $\gT$ remains almost the same as in Section
\ref{Sub:Finite}, with few straightforward changes. Namely, we drop
the requirement that $u<v$ when we consider intersection points
$d=\g(v)=\g_T(u)$ of $\g$ with $\g_T$, and the corresponding loop
$\tau_{d,T}(\g)$ consists of just two arcs: a piece of $\g$
parameterized by the arc $\overline{vu}$ of the circle, and
$\Traj_t(\g(u))$ for $0 \le t \le T$. See Figure \ref{Fig:NBased}b
(the loop $\tau_{d,T}(\g)$ is shown in bold). The following theorem
is an analogue of Theorem \ref{Th:Finite} for curves without base
points:

 \begin{thm}\label{Th:NBFinite}
Let $\g: S^1 \to \S$ be a generic immersed curve without the base point.
Let $X$ be a vector field on $\S$ which does not vanish on $\g$. Suppose
that $T>0$ is such that all intersections of $\g$ with $\g_T$ are
transversal double points. Then
 \begin{equation}\label{Eq:NBFinite}
\langle \g \rangle = \gT - \windex(\g,X).
 \end{equation}
 \end{thm}

Thus $\langle \g \rangle$ provides a simple obstruction for
pushing a curve off itself:
 \begin{cor}
Let $\g$ be a generic curve on $\S$. If $\langle \g
\rangle\notin\Z[[\g]-1]\subset\Z[\gO]$ then the curve $\g$ cannot be pushed
off itself by a flow of vector field, i.e.\ any shifted copy $\g_T$
intersects the initial curve $\g$.
 \end{cor}

 \begin{ex}\label{Ex:Pushoff}
Let $\S=\S_{0,3}$ be a doubly-punctured plane and $\g$ be a
figure-eight curve with lobes going around the punctures of $\S$.
Then $\g$ cannot be pushed off itself.
 \end{ex}

Note also that the only term in formula \eqref{Eq:NBFinite}, which
depends on $T$, is $\gT$. Thus it is, in fact, independent of $T$
and different values of $T$ give the same value of $\gT$.

\section{Proofs of Theorems \ref{Th:Finite} and \ref{Th:NBFinite}}
 \label{Sec:Proofs}

\subsection{Idea of the proofs}
The main idea is rather simple. We interpret both sides of the
formula \eqref{Eq:Finite} as two different ways to compute an
intersection number of two 2-chains in a 4-manifold. Very roughly
speaking, the manifold is $\S\times\S$ and two chains are
$\{(\g(u),\g(v)) \mid \ 0\le u\le v\le 1\}$ and the diagonal
$\{(x,x) \mid x\in\S\}$. Their intersection is the left part of the
formula (the sum of signs of double points). The same intersection
number can be also computed by intersecting the boundary of the
first chain with a 3-chain, constructed by homotopy
$\{(x,\Traj_t(x)) \mid \ 0\le t\le T\}$ of the diagonal. This
comprises the right hand side of the formula. The reality is
somewhat more complicated so some technicalities are involved. In
particular, to split the formula by homotopy classes $\tau\in\pi$ we
have to work in the universal covering space $\uS$ of $\S$. Also, to
push the boundary of the first chain off the diagonal, we take a
certain $\eps$ cut-off. Theorem \ref{Th:NBFinite} is completely
similar, except that we use a slightly different configuration
space. We treat the case of Theorem \ref{Th:Finite} in details and
indicate modifications needed for Theorem \ref{Th:Finite}.

\begin{rem}
An alternative combinatorial proof can be given as follows:
note that the equality is true for sufficiently small $T$.
Then let the flow of the vector field proceed, and check that the equality is preserved under any singular occurrence (i.e., a tangency or a passage of
the base point through a strand, for different relative orientations of the strands and directions of motion). Slight perturbations of the field may be
needed in order for the singular occurrence to be generic.
This proof, however, while somewhat shorter (although the number of different
cases is quite significant), does not explain the topological nature of these
formulas, understanding of which was the primary goal in this paper.
\end{rem}

\subsection{Configuration spaces}
Denote by $\uS$ the set of homotopy classes of paths $\xi: [a,b] \to
\S$ such that $\xi(a) = p$\,; $\uS$ is a universal covering space
for $\S$ and inherits its orientation. The projection
$\proj:\uS\to\S$ maps a path $\xi$ into its final point $\xi(b)$.
The vector field $X$ can be lifted to the vector field
$\widetilde{X}$ on $\uS$ such that $\proj(\widetilde{\Traj}(a)) =
\Traj(\proj(a))$, where $\widetilde{\Traj}$ means the trajectory of
$\widetilde{X}$.

Consider a configuration space $C=\{(u,v) \mid \ 0<u<v<1\}$ of
(ordered) pairs of points on $[0,1]$. To compactify $C$, we pick a
small $\eps>0$ and take an $\eps$-cut-off:
 \begin{equation*}
C^\eps=\{(u,v) \mid \ \eps\le u<u+\eps\le v\le 1-\eps\}\subset C.
 \end{equation*}
Note that apart from compactifying $C$, the condition $u+\eps\le v$
used in this cut-off allows us to push the boundary of $C^\eps$ off
the diagonal $u=v$. Topologically, $C^\eps$ is a closed 2-simplex,
whose boundary consists of three intervals, on which $u=\eps$,
$v=1-\eps$, and $u+\eps=v$, respectively.

The curve $\g$ defines an evaluation map $\ev: C \to \S \times \S$,
$\ev(u,v) \bydef (\g(u),\g(v))$ --- we forget the curve itself and
leave only the two marked points. For a fixed class $\tau\in\pi$
define $\ev_\tau: C^\eps \to \uS\times \uS$ by
 \begin{equation*}
\ev_\tau:(u,v)\mapsto (\xi_1,\xi_2)
 \end{equation*}
where $\xi_1$ is an arc $\g(t)$ with $t\in[0,u]$ (so its final point
is $\g(u)$) and $\xi_2$ is made of the loop $\tau$ followed by an
arc $\g(1-t)$ with $t\in[0,1-v]$ (so its final point is $\g(v)$),
see Figure \ref{Fig:Lift}. From the definition of $\ev_\tau$ we
immediately get $\proj \circ \ev_\tau(u,v)=\ev(u,v)$, so $\ev_\tau$
is a lift of the evaluation map.

\fig{height=1.2in}{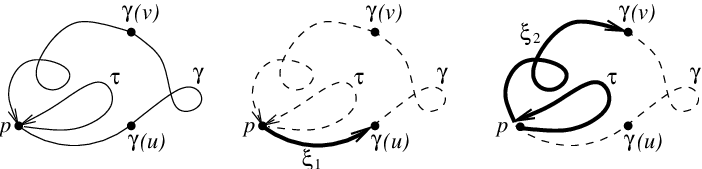}{Lifting the evaluation map}{Fig:Lift}

The diagonal
 \begin{equation*}
\Delta=\{(x,x) \mid x\in\uS\}
 \end{equation*}
is a proper 2-dimensional submanifold of $\uS \times \uS$, which
inherits the orientation of $\uS$ (as the image of the map $x\mapsto
(x,x)$). For a sufficiently small $\eps$ the intersection number
$\I(\ev_\tau(C^\eps),\Delta)$ of the image of $C^\eps$ with $\Delta$
is well-defined and independent of $\eps$, since $\ev_\tau(C^\eps)$
is oriented, compact, and its boundary does not intersect $\Delta$.

Denote by $\langle \g \rangle^\tau$ the coefficient of $\tau$ in
$\langle \g \rangle$, see \eqref{Eq:Sum}.
 \begin{lem}\label{Lem:Intersec}
We have $\I(\ev_\tau(C^\eps),\Delta)=\langle \g \rangle^\tau$
 \end{lem}
 \begin{proof}
Indeed, a pair $(u,v)$ gives an intersection point of
$\ev_\tau(C^\eps)$ with $\Delta$ iff the homotopy classes of paths
$\xi_1$ and $\xi_2$ coincide, i.e., endpoints $\g(u)$ and $\g(v)$ of
these paths coincide and the homotopy class of
$\xi_2^{-1}\circ\xi_1$ is trivial. This means that $d=\g(u)=\g(v)$
is a self-intersection point of $\g$ and the homotopy class of the
loop $\g(t)$ with $t\in[0,u]\cup[v,1]$ equals to $\tau$. The local
intersection sign is easy to compute and a direct check assures that
orientations of $\ev_\tau(C^\eps)$ and $\Delta$ give the positive
orientation of $\uS\times\uS$ iff $\sgn(d)=+1$.
 \end{proof}

Consider the $2$-chain $\Delta_T \bydef \{(x,
\widetilde{\Traj}_T(x)) \mid x \in \uS\}$. The homotopy $W_T =
\{(x,\widetilde{\Traj}_t(x)) \mid 0 \le t \le T\}$ is a $3$-chain
such that $\partial W_T = \Delta_T - \Delta$. Denote $\I$ the
intersection number, so that one has
 \begin{equation}
 \label{Eq:Intersec}
\I(\ev_\tau(C^\eps),\Delta)=\I(\ev_\tau(C^\eps),\Delta_T)-\I(\ev_\tau(\partial
C^\eps),W_T)
 \end{equation}

Denote by $\gT^\tau$ the coefficient of $\tau$ in $\gT$, see
\eqref{Eq:SumT}.
 \begin{lem}\label{Lem:IntersecT}
We have $\I(\ev_\tau(C^\eps),\Delta_T)=\gT^\tau$
 \end{lem}
\noindent The proof repeats the proof of Lemma \ref{Lem:Intersec}
above.

In view of \eqref{Eq:Intersec} it remains to compute
$\I(\ev_\tau(\partial C^\eps),W_T)$. Each intersection point of
$\ev_\tau(\partial C^\eps)$ with $W_T$ corresponds to a pair
$(u,v)\in\partial C^\eps$, such that the endpoint of the
corresponding path $\xi_2$ is obtained from the endpoint of $\xi_1$
by a diffeomorphism $\Traj_t$ for some $0\le t\le T$ (in other
words, $\Traj_t(\g(u))=\g(v)$) and the homotopy class of the path
$\xi_2^{-1}\circ\Traj_t(\g(u))\circ\xi_1$ is $\tau$. Let us study
separately each of the three parts of the boundary $\partial
C^\eps$. Denote by $\partial_-$ the part of $\partial C^\eps$ on
which $u=\eps$, by $\partial_+$ the part on which $v=1-\eps$, and
$\partial_=$ the part on which $u+\eps=v$ (with the induced
orientation).

Denote by $\langle \g_1,\g_2 \rangle^\tau$ the coefficient of $\tau$
in $\langle \g_1,\g_2 \rangle$, see \eqref{Eq:SumGG}.
 \begin{lem}\label{Lem:IntersecBdry1}
We have $\I(\ev_\tau(\partial_-),W_T)=\langle \g,\Traj_-
\rangle^\tau$ and $\I(\ev_\tau(\partial_+),W_T)=\langle \Traj_+,\g
\rangle^\tau$. Thus $\sum_\tau
\I(\ev_\tau(\partial_-),W_T)+\I(\ev_\tau(\partial_+),W_T)=2\indT$
 \end{lem}
 \begin{proof}
For $(u,v)\in\partial_-$ we have $u = \eps$ and $2\eps\le v\le
1-\eps$, thus we are interested in points of intersection of
$\Traj_t(\g(\eps))$ with $\g(v)$ for $0\le t\le T$ and $2\eps\le
v\le 1-\eps$. But for small $\eps$ these are just intersections of
$\Traj_+$ (see Section \ref{Sub:Finite}) with the whole of $\g$.
Moreover, it is easy to check that the signs of such intersection
points coincide with the signs $\sgn(d)$, $d\in D(\Traj_+,\g)$
introduced in Section \ref{Sub:Compose}, which proves the first
equality of the lemma. The second equality is proven in the same
way, after we notice that $\eps\le u\le 1-2\eps$ and $v=1-\eps$
correspond to intersections of $\Traj_t(\g(u))$ with $\g(v)=p$,
i.e., of $\g$ with $\Traj_-$.
 \end{proof}

 \begin{lem}\label{Lem:IntersecBdry2}
We have $\I(\ev_{[\g]}(\partial_=),W_T)=w(\g,X)$ and
$\I(\ev_\tau(\partial_=),W_T)=0$ for $\tau\ne [\g]$.
 \end{lem}
 \begin{proof}
For $(u,v)\in\partial_=$ we have $u+\eps=v$. For small $\eps$,
values of $u$ for which $\Traj_t(\g(u))=\g(u+\eps)$ correspond to
points in which $\tg$ looks in the direction of $X$. Each such point
gives the homotopy class $\tau=[\g]$ and is counted with a positive
sign, if $\tg$ turns counter-clockwise relative to $X$ in a
neighborhood of $\g(t)$ and with a negative sign otherwise.
Comparison with the definition of $w(\g,X)$ in Section
\ref{Sec:Selfintersec} proves the lemma.
 \end{proof}

This concludes the proof of Theorem \ref{Th:Finite}.

Finally, to prove Theorem \ref{Th:NBFinite} we need another
configuration space. Call $\Ann$ the set of pairs of points $(u,v)$
on a circle such that the length of the arc $\overline{uv}$ and the
length of the arc $\overline{vu}$ are both $\ge \eps$.
Topologically, $\Ann$ is an annulus, and its boundary $\partial\Ann$
is a union of two circles $L_+$ and $L_-$ defined by the conditions
$\lmod \overline{vu}\rmod = \eps$ and $\lmod \overline{uv}\rmod =
\eps$, respectively. The rest of the proof copies the proof of
Theorem \ref{Th:Finite}.

\end{document}